\documentclass[11pt]{amsart}
\usepackage{amssymb}
\theoremstyle{plain}

\newcommand{\N}{\mathbb{N}}

\newcommand{\R}{\mathbb{R}}

\newcommand{\vp}{\varepsilon}

\newtheorem{thm}{Theorem}
\newtheorem*{thmA}{Theorem A}
\newtheorem*{thmC}{Theorem C}
\newtheorem{lem}[thm]{Lemma}

\newtheorem{prop}[thm]{Proposition}
\newtheorem*{propB}{Proposition B}
\theoremstyle{definition}

\newtheorem*{defn}{Definition}          %% UN-NUMBERED
\newtheorem{prob}[thm]{Problem}

\newcommand{\supp}{{\rm supp}}
\newcommand{\coo}{c_{00}}

\newcommand{\kleq}{\!\leq\!}

\newcommand{\kin}{\!\in\!}

\def\hangbox to #1 #2{\vskip1pt\hangindent #1\noindent \hbox to #1{#2}$\!\!$}

\begin{document}

\title{Greedy bases for Besov spaces}

\author {S. J. Dilworth}\address{Department of Mathematics \\
University of South Carolina\\ Columbia, SC 29208, USA}
\email{dilworth@math.sc.edu}
\author{D. Freeman}\address{Department of Mathematics \\  Texas A\&M University\\
College Station, TX 77843, USA}\email{freeman@math.tamu.edu}
\author{E. Odell}
\address{Department of Mathematics \\
The University of Texas\\1 University Station C1200\\
Austin, TX 78712,  USA} \email{odell@math.utexas.edu}
\author{Th. Schlumprecht}
\address{Department of Mathematics, Texas A\&M University\\
College Station, TX 77843, USA}
\email{thomas.schlumprecht@math.tamu.edu}
\renewcommand{\subjclassname}{\textup{2000} Mathematics Subject
Classification}
\thanks{\textit{2000 Mathematics Subject Classification}: Primary 46B15, Secondary 41A65.} 
\thanks{The research of
the  first, third, and fourth  authors was supported by   National Science Foundation grants 
DMS 0701552,
DMS 0700126, and
DMS 0856148, respectively.
 The second author 
was supported by   grant  N000140811113 of the Office of Naval Research.
The first and third authors were  also supported by the Workshop in
Analysis and Probability at Texas A\&M University in  2009. }
\keywords{Greedy bases; Besov spaces} 

\begin{abstract} We prove that
the Banach space $(\oplus_{n=1}^\infty \ell_p^n)_{\ell_q}$, which is isomorphic to certain Besov spaces,    has a
greedy basis whenever $1\leq p \leq\infty$ and $1<q<\infty$.
Furthermore, the Banach spaces $(\oplus_{n=1}^\infty \ell_p^n)_{\ell_1}$,  with $1<p\le \infty$, and
$(\oplus_{n=1}^\infty \ell_p^n)_{c_0}$, with $1\le p<\infty$  do not have a greedy bases.
We prove as well that the space $(\oplus_{n=1}^\infty \ell_p^n)_{\ell_q}$ has a 1-greedy basis
if and only if $1\leq p=q\le \infty$.
\end{abstract}

\maketitle

\allowdisplaybreaks
\section{Introduction}\label{S:1}
Let $X$ be a  a Banach space and let $(x_i)$ be a Schauder basis
for $X$ with biorthogonal sequence $(x_i^*)$. For $x \in X$ and
$n\ge1$, the error in the best $n$-term approximation to $x$
(using $(x_i)$) is given by
$$\sigma_n(x) := \inf \Big\{\Big\|x - \sum_{i \in A} a_i x_i\Big\| \colon
(a_i) \subset \mathbb{R}, |A| \le n\Big\}.$$
 Let $A_n(x)\subset
\mathbb{N}$ be the indices corresponding to a choice of   $n$   largest coefficients of
$x$ in absolute value, i.e. $A_n(x)$ satisfies
$$ \min \big\{|e_i^*(x)| \colon i \in A_n(x)\big\} \ge \max  \big\{
|e_i^*(x)| \colon i \in \mathbb{N} \setminus A_n(x)\big\}.$$ 
Then
$G_n(x) := \sum_{i \in A_n(x)} x^*_i(x) x_i$ is called an $n^{th}$
\textit{greedy approximant} to $x$. We say that  $(x_i)$ is \textit{greedy} with
constant $C$ if
$$ \|x - G_n(x)\| \le C \sigma_n(x) \qquad(x \in X, n\ge1).$$
If $C=1$ then $(x_i)$ is said to be $1$-greedy. Temlyakov \cite{T1} proved
that the Haar system for $L_p[0,1]^d$ ($1<p<\infty$, $d\ge1$) is
greedy, which provides
an important theoretical justification for the thresholding
procedure used in data compression.  Subsequently, Konyagin and
Temlyakov \cite{KT} gave a very useful abstract characterization of greedy bases.
To state their result, we recall that $(x_i)$ is \textit{unconditional}
with constant $K$ if, for all choices of signs, we have
$$ \Big\|\sum_{i=1}^\infty \pm x^*_i(x)x_i \Big\| \le K\|x\| \qquad(x \in X).$$ 
We say that $(x_i)$ is \textit{democratic} with constant $\Delta$ if,
for all finite  $A, B \subset \mathbb{N}$ with $|A| = |B|$, we
have
$$\Big\|\sum_{i \in A} x_i\Big\| \le \Delta \Big\|\sum_{i \in B} x_i\Big\|.$$

\begin{thmA}\label{T:1} {\rm  \cite{KT}} Suppose that $(x_i)$ is
unconditional with constant $K$ and democratic with constant
$\Delta$. Then $(x_i)$ is greedy with constant $K + K^3\Delta$.
Conversely, if $(x_i)$ is greedy with constant $C$ then $(x_i)$ is
unconditional with constant $C$ and democratic with constant
$C^2$.
\end{thmA}

Theorem~A was used in \cite{W1, W2} to prove that
$L_p[0,1]$ ($p \ne 2$)  has a greedy basis that is not equivalent
to a subsequence of the Haar basis, and in \cite{DHK} to prove that
$\ell_p$  and $L_p[0,1]$ ($p \ne 2$) have a continuum of mutually
non-equivalent greedy bases.  It was also used in \cite{DKKT} to study
duality for greedy bases, and a similar  theorem was proved in \cite{DKKT} to
characterize the  larger class of \textit{almost greedy} bases (see
also \cite{DKK}).

Some examples of greedy bases are given in \cite{W2}. In most
cases these bases  are  greedy simply because  they are
\textit{symmetric} (e.g. Riesz bases for a Hilbert space, which
are equivalent to the unit vector basis of $\ell_2$, or good
wavelet bases for the Besov spaces $B^{p}_{\alpha,p}(\mathbb{R})$,
which are equivalent to the unit vector basis of $\ell_p$), or
because they are equivalent to the Haar basis (e.g. good wavelet
bases for $L_p(\mathbb{R}^d)$) or to a subsequence of the Haar
basis (e.g. generalized Haar systems \cite{K}).  In \cite{GH}
certain wavelet bases in the  Triebel-Lizorkin spaces $\mathfrak{f}_{p,r}^s$ are
shown to be greedy. In \cite{AW} it is  proved that $1$-symmetric
bases (e.g. the unit vector bases of Orlicz and Lorentz sequence
spaces) are in fact $1$-greedy. On the other hand, there are
examples of spaces with an unconditional basis but no democratic
unconditional basis, and hence no greedy basis, e.g. certain
spaces with a unique unconditional basis up to permutation
\cite{BCLT}, the spaces $\ell_p \oplus \ell_q$ and $\ell_p \oplus
c_0$ for $1 \le p < q < \infty$ \cite{AW},   and the original
Tsirelson space $T^*$ \cite{T2}. Wojtaszczyk \cite{W3} proved that
the $L_p$ spaces ($1<p<\infty$) are the only
rearrangement-invariant function spaces on $[0,1]$ for which the
Haar system  is greedy.

Using Theorem~A we prove that for every $1\leq p \leq\infty$ and $1<q<\infty$
the Banach space $(\oplus_{n=1}^\infty \ell_p^n)_{\ell_q}$ has a
greedy basis.  Furthermore, we
show that  the Banach space $(\oplus_{n=1}^\infty
\ell_p^n)_{\ell_1}$ does not have a greedy basis whenever
$1<p\leq\infty$. This answers a question posed by P. Wojtaszczyk,
who asked when such spaces have a greedy basis.
The problem of
finding a greedy basis for Banach spaces of the form
$(\oplus_{n=1}^\infty \ell_p^n)_{\ell_q}$ is particularly
pertinent from the approximation theoretical standpoint as such
spaces are isomorphic to certain Besov spaces on the circle  \cite{R}. 
 As shown in \cite[Theorem  2]{R}  (see also \cite[page 255]{Pi}), the Besov space $B^{\alpha, m}_{p,q}[0,1]$,
 where $1\le p\le \infty$, $1\le q<\infty$, $m\in\{-1,0,1,2\ldots\}$, and $\alpha\in(1,m+1+1/p)$,
 is isomorphic  to $\ell_1^{m+2}\oplus (\oplus_{n=0}^\infty \ell_p^{2^n})_{\ell_q}$ which is easily seen to be
 isomorphic to $ (\oplus_{n=1}^\infty \ell_p^n)_{\ell_q}$.

The
greedy bases which we construct in Theorem~\ref{T:2} differ from the examples discussed above in that they are neither
subsymmetric  (see \cite[p. 114]{LT} for this notion)
nor equivalent to a subsequence of the Haar basis.

The following result completely characterizes  for which pairs
$(p,q)$  the space
$(\oplus_{n=1}^\infty \ell_p^n)_{\ell_q}$ has a greedy basis.

\begin{thm}\label{T:2} Let $1\le p,q\le \infty$.
\begin{enumerate}
\item[a)] If $1<q<\infty$ then the Banach space
$(\oplus_{n=1}^\infty \ell_p^n)_{\ell_q}$  has a greedy basis.
\item[b)]  The spaces   $(\oplus_{n=1}^\infty \ell_p^n)_{\ell_1}$,
with $1<p\le \infty$, and
  $(\oplus_{n=1}^\infty \ell_p^n)_{c_0}$, with $1\le p< \infty$,  do not have greedy bases.
\end{enumerate}
\end{thm}

The following result yields that only in the trivial case that
$p=q$   does $(\oplus_{n=1}^\infty \ell_p^n)_{\ell_q}$  have a
$1$-greedy basis.
\begin{thm}\label{T:2a}
Let $1\leq p\le \infty$ and let $(E_n)_{n=1}^\infty$ be a sequence
of finite dimensional  Banach spaces. If $(x_i)_{i=1}^\infty$ is a
normalized 1-greedy basis for the space $(\oplus_{n=1}^\infty
E_n)_{\ell_p}$ then $(x_i)_{i=1}^\infty$ is 1-equivalent to the
standard unit vector basis for $\ell_p$ (as usual, if $p=\infty$
we consider the $c_0$-sum).
\end{thm}

As the cases $\ell_p\oplus \ell_q$ and $(\oplus_{n=1}^\infty \ell_p^n)_{\ell_q}$ are settled 
 the following spaces might be interesting to consider.
\begin{prob}  Assume $1<p\not= q<\infty$. Does $\ell_q(\ell_p)=(\oplus \ell_p)_{\ell_q}$ have a greedy basis?
\end{prob}
We thank P. Wojtaszczyk for  his comments on a preliminary version of  our paper  and for  drawing \cite{R}  to our notice.

\section{Proof of  Theorems \ref{T:2} and \ref{T:2a} }\label{S:2}

Part (a) of  Theorem \ref{T:2}  will   follow  easily from the following Lemma, whose proof will require some work.

\begin{lem}\label{L:1}
Let $1\leq p \leq\infty$ and $1<q<\infty$ and let $\vp>0$.
There is a constant $1\leq K<\infty$ such that for all $N\in\N$
there exist $M=M_N$ and  a finite normalized  sequence
$(x_i)_{i=1}^M\subset \ell_q(\ell_p^N)$ such that
\begin{enumerate}
\item[a)]  $(x_i)_{i=1}^M$ is  $1$-unconditional,
\item[b)] $(1-\vp)|A|\leq\Big\|\sum_{i\in A}x_i\Big\|^q\leq(1+\vp)|A|$\quad for all $A\subset\{1,...,M\}$,
\item[c)] the span of $(x_i)_{i=1}^M$ is $K$-complemented in $\ell_q(\ell_p^N)$, and
\item[d)] $\ell_p^N$ is isometric to a $K$-complemented subspace of
the span of $(x_i)_{i=1}^M$.
\end{enumerate}
\end{lem}

Using the lemma, we  give a quick proof of the first part of  Theorem \ref{T:2}.
\begin{proof}[Proof of  Theorem \ref{T:2} (a)] Let  $1\le p \le \infty$ and $1<q<\infty$.
It will be more convenient for us to work with the space  $X:=(\oplus_{N=1}^\infty\ell_q(\ell_p^N))_{\ell_q}$  instead
of    $(\oplus_{n=1}^\infty \ell_p^n)_{\ell_q}$. That these spaces are isomorphic follows from  {\em  Pe\l czy\'nski's Decomposition
Method} \cite{Pe}, which says that if two Banach spaces  are complementably embedded in each other, and one of them is isomorphic
to the (countably infinite) $\ell_r$-sum, $1\le r<\infty$,  or $c_0$-sum of itself, then  they are isomorphic. It is easy to observe that $X$ and  $(\oplus_{n=1}^\infty \ell_p^n)_{\ell_q}$
are 1-complemented in each other, and that $X$ is isometric to its $\ell_q$ sum.

We let $\vp>0$ and choose, for each $N\in\N$, a sequence $(x^{(N)}_i)_{i=1}^{M_N}$ in the $N$th cordinate of $X=(\oplus_{N=1}^\infty\ell_q(\ell_p^{(N)}))_{\ell_q}$ which satisfies
Lemma \ref{L:1}.  From the conditions (a) and (b) in Lemma \ref{L:1}
it follows that  $(x^{(N)}_i)_{N\in\N,1\leq i \leq M_N}$ is a  1-unconditional and $\frac{1+\vp}{1-\vp}$-democratic sequence in $X$.  As the span of
$(x^{(N)}_i)_{i=1}^{M_N}$ is $K$-complemented in $\ell_q(\ell_p^N)$, for each $N\in\N$, it follows that the closed span of $(x^{(N)}_i)_{N\in\N,1\leq i \leq M_N}$ is $K$-complemented
in $X$.  Furthermore, as $\ell_p^N$ is $K$-complemented in the span of $(x^{(N)}_i)_{i=1}^{M_N}$ for each $N\in\N$, it
follows that $(\oplus_{N=1}^\infty\ell_p^N)_{\ell_q}$ is $K$-complemented in the closed span of $(x^{(N)}_i)_{N\in\N,1\leq i \leq M_N}$.  Thus by the
Pe\l czy\'nski decomposition theorem, $X$ is isomorphic to the closed span of $(x^{(N)}_i)_{N\in\N,1\leq i \leq M_N}$.  Hence
$X$, and thus $(\oplus_{n=1}^\infty \ell_p^n)_{\ell_q}$, has a greedy basis.
\end{proof}

\begin{proof}[Proof of Lemma \ref{L:1}]
For $\vp>0$, we choose numbers
$\vp_i\searrow0$ such that $\prod_{i=1}^\infty(1+\vp_i)<1+\vp$ and
$\prod_{i=1}^\infty(1-\vp_i)>1-\vp$. For each $n\in\N$, we denote  by $(e_{(i,n)})^{N}_{i=1}$   the unit vector basis for the $n$th coordinate of
$\ell_q(\ell_p^N)$, and
we denote by $(e^*_{(i,n)})^{N}_{i=1}$ their biorthogonal functionals.
Thus  the norm on $\ell_q(\ell_p^N)$ is calculated by
$$\Big\|\sum
a_{(i,n)}e_{(i,n)}\Big\|=\Big(\sum_{n=1}^\infty\Big(\sum_{i=1}^N|a_{(i,n)}|^p\Big)^{q/p}\Big)^{1/q} \text{ for $(a_{(i,n)}:1\kleq i\kleq N, n\kin\N)\subset \R$}. $$
If $x\in\ell_q(\ell_p^N)$, then we denote the support of $x$ by $\supp(x)=\{(i,n) |
e^*_{(i,n)}(x)\neq0\}$.

Before proceeding we fix two sequences of integers
$(m_i),(k_i)\in[\N]^{\omega}$ with $m_1=k_1=1$ and which satisfy the
following inequalities for all $i>1$.
\begin{enumerate}
\item[i)] $1/\vp_i<m_i$,
\item[ii)] $m_i^{1/q}+1<(1+\vp_i)^{1/q}m_i^{1/q}$,
\item[iii)] $(1+(m_i/k_i)^{1/q})^q<1+\vp_i$, and
\item[iv)] $1-\vp_i<(1-(m_i/k_i)^{1/q})^q$.
\end{enumerate}
The above inequalities can be easily guaranteed by first choosing
$m_i$ large enough to satisfy $i)$ and $ii)$, and  then choosing $k_i$
large enough to satisfy $iii)$ and $iv)$.  For the sake of
convenience we define $n_j=\prod_{i=1}^j k_i$ for all $j\in\N$.  We
define the finite family
$(x_{(i,j)})_{1\leq i\leq N, 1\leq j\leq n_N/n_i}$ by
$$x_{(i,j)}=\frac{1}{n_i^{1/q}}\sum_{s=1}^{n_i}e_{(i,s+(j-1)n_i)}\quad\text{for all } 1\leq i\leq N\textrm{ and }1\leq j\leq n_N/n_i.
$$
It is clear that $(x_{(i,j)})$ is a normalized and 1-unconditional
basic sequence as the sequence has pairwise disjoint support. Also,
$\bigcup_{i\le N, j\le n_N/n_i} \supp (x_{(i,j)})= \{1,2,\ldots, N\}\times \{1,2,\ldots, n_N\}$.
For each integer $1\leq \ell\leq N$ and subset
$$A\subset\{(i,j)\in\N^2|1\leq i\leq\ell \textrm{ and }1\leq j\leq
n_N/n_i\},$$ we will prove by induction on $\ell$ that
\begin{equation}\label{E:3}
|A|\prod_{i=2}^\ell(1-\vp_i)\leq \Big\|\sum_{(i,j)\in A}x_{(i,j)}\Big\|^q\leq|A|\prod_{i=2}^\ell(1+\vp_i).\end{equation}
First note
that if $\ell=1$ and $A\subset\{(1,j)\in\N^2|1\leq j\leq n_N\}$ then
$$\Big\|\sum_{(1,j)\in A}x_{(1,j)}\Big\|^q=\Big\|\sum_{(1,j)\in A}e_{(1,j)}\Big\|^q=|A|.$$
Thus  \eqref{E:3} is trivially satisfied.  We now assume that
equation \eqref{E:3} is satisfied for a given $1\leq\ell<N$ and we will
prove it for $\ell+1$.  We first partition the set $\Omega=\{(i,j)\in\N^2 : 1\leq
i\leq\ell+1 \textrm{ and }1\leq j\leq n_N/n_{i}\}$, 
into sets $\Omega_1,\Omega_2\ldots, \Omega_{n_N/n_{\ell+1}}$ defined for each $1\kleq r\kleq n_N/n_{\ell+1}$ by
$$\Omega_r:= \{(i,j)\in\N^2|1\leq i\leq\ell+1,
(r-1)n_{\ell+1}/n_i+1\leq j\leq r n_{\ell+1}/n_i\}.$$
Observe that for $1\kleq r\kleq n_N/n_{\ell+1}$, $1\kleq i\kleq \ell+1$, and $1\kleq j\kleq n_N/n_i$
\begin{align*} (i,j)\in \Omega_n&\iff (r-1)\frac{n_{\ell+1}}{n_i}+1\kleq j\kleq r \frac{n_{\ell+1}}{n_i}\\
                                          &\iff \supp(x_{(i,j)}) = \{i\}\times[(j-1)n_i+1,jn_i]\subset \{i\}\times [(r-1) n_{\ell+1}+1,r n_{\ell+1}].
\end{align*}
Since $\supp(x_{(\ell+1,r)})=\{\ell+1\}\times [(r-1) n_{\ell+1}+1,r n_{\ell+1}]$,
it follows that
$$(i,j)\in \Omega_r\iff \{(\ell+1, s): (i,s)\in\supp(x_{(i,j)})\}\subset \supp(x_{(\ell+1, r)}).$$
Given the set $A\subset\Omega$, we partition $A$ into $A_1$, $A_2$, ... $A_{n_N/n_{\ell+1}}$, by defining $A_r=A\cap \Omega_r$,
 for all $1\le r\le n_N/n_{\ell+1}$. 
We note that
$$\Big\|\sum_{(i,j)\in A}x_{(i,j)}\Big\|^q=\sum_{r=1}^{n_N/n_{\ell+1}}
\Big\|\sum_{(i,j)\in A_r}x_{(i,j)}\Big\|^q$$
and that $(x_{(i,j)})_{(i,j)\in\Omega_1}$ is 1-equivalent to $(x_{(i,j)})_{(i,j)\in\Omega_r}$, for all $1\le r\le n_N/n_{\ell+1}$.  Thus, to prove the
inequality (\ref{E:3}), we just need to consider the case $A=A_1$. We
first note that if $(\ell+1,1)\not\in A_1$, then the inequality
\eqref{E:3} is immediately true by the induction hypothesis.  Thus
we now assume that $(\ell+1,1)\in A_1$ and $A_1\setminus\{(\ell+1,1)\}\not=\emptyset$ .

Roughly speaking, we will argue that either $|A_1|$ is large enough so that
$\sum_{(i,j)\in A_1} x_{(i,j)}$ can be replaced by   $\sum_{(i,j)\in A_1\setminus\{\ell+1,1\}} x_{(i,j)}$,
or  $|A_1|$ is so small that a large part of the support  of $x_{(\ell+1,1)}$
is disjoint from
$$B_1= \big\{(\ell+1,n): (i,n)\in  \supp(x_{(i,j)}) \text{ for some $j\kin\{1,2\ldots n_N/n_i\}$}\},$$ and we can approximate $x_{(\ell+1,1)}$ by its projection
onto  $\text{\rm span}(e_{(\ell+1,n)}: n\kin\{1,2,\ldots,N\}\setminus B_1)$.

The first case we consider is
that $|A_1|\geq m_{\ell+1}$.  This assumption, together with the
inquality $m_{\ell+1}>1/\vp_{\ell+1}$, yields
$$\frac{|A_1|}{|A_1|-1}\leq
\frac{m_{\ell+1}}{m_{\ell+1}-1}<\frac{1}{1-\vp_{\ell+1}}.
$$
This allows us to obtain the desired  lower estimate. Indeed,
$$
\Big\|\sum_{(i,j)\in A_1} x_{(i,j)}\Big\|
\geq\Big\|\sum_{(i,j)\in A_1\setminus\{(\ell+1,1)\}} x_{(i,j)}\Big\|
\geq(|A_1|-1)\prod_{i=2}^\ell(1-\vp_i)\geq |A_1|\prod_{i=2}^{\ell+1}(1-\vp_i).
$$
To prove the upper estimate in (\ref{E:3}), we use that $|A_1|\geq
m_{\ell+1}$ together with $ii)$ to get
$$\frac{(|A_1|-1)^{1/q}+1}{|A_1|^{1/q}}<\frac{m_{\ell+1}^{1/q}+1}{m_{\ell+1}^{1/q}}<(1+\vp_{\ell+1})^{1/q}\!.
$$
Thus,
\begin{align*}
\Big\|\sum_{(i,j)\in A_1}x_{(i,j)}\Big\|&\leq\Big\|\sum_{(i,j)\in A_1\setminus{(\ell+1,1)}}x_{(i,j)}\Big\|+1\\
 &\leq(|A_1|\!-\!1)^{1/q}\Big(\prod_{i=2}^\ell(1+\vp_i)\Big)^{1/q}+\!\Big(\prod_{i=2}^\ell(1+\vp_i)\Big)^{1/q}
 <|A_1|^{1/q} \Big(\prod_{i=2}^{\ell+1}(1\!+\!\vp_i)\Big)^{1/q} .
\end{align*}

This completes the proof of  \eqref{E:3} for $\ell+1$ in the case
that $|A_1|\geq m_{\ell+1}$.  We now assume that $|A_1|< m_{\ell+1}$.  The
size of the support of each $x_{(i,j)}$ is given by
$|\supp(x_{(i,j)})|=n_i$ for all $1\leq i\leq \ell+1$ and $1\leq j\leq
n_{N}/n_{i}$.  We thus have a simple estimate for the size of the
union of the supports,
\begin{equation}\label{E:1a} |\cup_{(i,j)\in A_1\setminus\{(\ell+1,1)\}}\supp(x_{(i,j)})|=\sum_{(i,j)\in
A_1\setminus\{(\ell+1,1)\}}|\supp(x_{(i,j)})|<|A_1|n_\ell.
\end{equation}
We define sets
\begin{align*}
&B_1:=\{(\ell+1,n)\in\N^2 :(m,n)\in\cup_{(i,j)\in
A_1\setminus\{(\ell+1,1)\}}\supp(x_{(i,j)})\textrm{ for some }1\leq m\leq \ell\}
\text{ and }\\
&B_2:=\supp(x_{(\ell+1,1)})\setminus B_1.\end{align*}
The  inequality  \eqref{E:1a} gives
that $|B_1|<|A_1|n_\ell$.  For $i=1,2$, we define $P_{B_i} x_{(\ell+1,1)}$ by
$P_{B_i}x_{(\ell+1,1)}=\frac{1}{n_{\ell+1}^{1/q}}\sum_{(\ell+1,j)\in
B_i}e_{(\ell+1,j)}$. We may estimate the value $\|P_{B_1}x_{(\ell+1,1)}\|$ by
\begin{equation}\label{E:1b}
\|P_{B_1}x_{(\ell+1,1)}\|=\frac{1}{n_{\ell+1}^{1/q}}|B_1|^{1/q}<\Big(\frac{|A_1|n_\ell}{n_{\ell+1}}\Big)^{1/q}<
\Big(\frac{m_{\ell+1} n_\ell}{n_{\ell+1}}\Big)^{1/q}=\Big(\frac{m_{\ell+1}}{k_{\ell+1}}\Big)^{1/q}.
\end{equation}
We  use this to obtain the following estimate.
\begin{align*}
\Big\|\sum_{(i,j)\in A_1}x_{(i,j)}\Big\|^q&=\Big\|P_{B_2}x_{(\ell+1,1)}+P_{B_1}x_{(\ell+1,1)}+\sum_{(i,j)\in A_1\setminus\{(\ell+1,1)\}}x_{(i,j)}\Big\|^q\\
&=\|P_{B_2}x_{(\ell+1,1)}\|^q+\Big\|P_{B_1}x_{(\ell+1,1)}+\sum_{(i,j)\in A_1\setminus\{(\ell+1,1)\}}x_{(i,j)}\Big\|^q\\
&\leq1 + \Big[\Big(\frac{m_{\ell+1}}{k_{\ell+1}}\Big)^{1/q}+\Big(\big(|A_1|-1)\prod_{i=1}^\ell(1+\vp_i)\Big)^{1/q}\Big]^q\\
&=1+\Big[ 1+\Big(\frac{m_{\ell+1}}{k_{\ell+1}}\Big)^{1/q}
                \Big((|A_1|-1)\prod_{i=1}^\ell(1+\vp_i)\Big)^{-1/q}\Big]^q      (|A_1|-1)\prod_{i=1}^\ell(1+\vp_i)\\
                &\le \Big[ 1+\Big(\frac{m_{\ell+1}}{k_{\ell+1}}\Big)^{1/q}
                \Big(\prod_{i=1}^\ell(1+\vp_i)\Big)^{-1/q}\Big]^q\prod_{i=1}^\ell(1+\vp_i)\\
                &\qquad + \Big[ 1+\Big(\frac{m_{\ell+1}}{k_{\ell+1}}\Big)^{1/q}
                \Big((|A_1|-1)\prod_{i=1}^\ell(1+\vp_i)\Big)^{-1/q}\Big]^q      (|A_1|-1)\prod_{i=1}^\ell(1+\vp_i)\\
&\le\Big[1+\Big(\frac{m_{\ell+1}}{k_{\ell+1}}\Big)^{1/q}\Big(\prod_{i=1}^\ell(1+\vp_i)\Big)^{-1/q}\Big]^q  |A_1|\prod_{i=1}^\ell(1+\vp_i)\\
&\le|A_1| \prod_{i=1}^{\ell+1}(1+\vp_i)\quad\textrm{ by iii)}.
\end{align*}
For proving the remaining lower inequality in (\ref{E:3}), we will use the
following estimate for $\|P_{B_2}x_{(\ell+1,1)}\|$ which follows from \eqref{E:1b}.
$$\|P_{B_2}x_{(\ell+1,1)}\|\geq
1-\|P_{B_1}x_{(\ell+1,1)}\|>1-\Big(\frac{m_{\ell+1}}{k_{\ell+1}}\Big)^{1/q}.
$$
This yields
\begin{align*}
\Big\|\sum_{(i,j)\in
A_1}x_{(i,j)}\Big\|^q&=\|P_{B_2}x_{(\ell+1,1)}\|^q+\Big\|P_{B_1}x_{(\ell+1,1)}+\sum_{(i,j)\in A_1\setminus\{(\ell+1,1)\}}x_{(i,j)}\Big\|^q\\
&\geq\|P_{B_2}x_{(\ell+1,1)}\|^q+\Big(-\|P_{B_1}x_{(\ell+1,1)}\|+\Big\|\sum_{(i,j)\in A_1\setminus\{(\ell+1,1)\}}x_{(i,j)}\Big\|\Big)^q\\
&=\|P_{B_2}x_{(\ell+1,1)}\|^q\\
&\quad+\Big(1-\|P_{B_1}x_{(\ell+1,1)}\|\cdot\Big\|\sum_{(i,j)\in A_1\setminus\{(\ell+1,1)\}}x_{(i,j)}\Big\|^{-1}\Big)^q\Big\|\sum_{(i,j)\in
A_1\setminus\{(\ell+1,1)\}}x_{(i,j)}\Big\|^q\\
&\geq\Big(1-\Big(\frac{m_{\ell+1}}{k_{\ell+1}}\Big)^{1/q}\Big)^q+ \Big(1-\Big(\frac{m_{\ell+1}}{k_{\ell+1}}\Big)^{1/q}\Big)^q (|A_1|-1) \prod_{i=1}^\ell(1-\vp_i)\\
&\geq  \Big(1-\Big(\frac{m_{\ell+1}}{k_{\ell+1}}\Big)^{1/q}\Big)^q |A_1| \prod_{i=1}^\ell(1-\vp_i)\\
&>\prod_{i=1}^{\ell+1}(1-\vp_i)|A_1|\quad\textrm{by iv)}.\\
\end{align*}

Thus we have proven the inequalities (\ref{E:3}) in all cases.  It
remains to prove that there exists a constant $1\leq K<\infty$,
independent of $N\in\N$, such that $X:=\text{\rm span}(x_{(i,j)})$ is
$K$-complemented in $\ell_q(\ell_p^N)$ and $\ell_p^N$ is isometric
to a $K$-complemented subspace of $X$.  For each $1\leq
i\leq N$, we define the vector $y_i$ as
$$y_i:=\frac{n_i^{1/q}}{n_N^{1/q}}\sum_{j=1}^{n_N/n_i}x_{(i,j)}=\frac{1}{n_N^{1/q}}\sum_{j=1}^{n_N}e_{(i,j)}.
$$
It should be clear that $(y_i)_{i=1}^N$ is 1-equivalent to the unit
vector basis for $\ell_p^N$.  Indeed, if $(a_i)\in\ell_p^N$ then
$$\Big\|\sum_{i=1}^N a_i y_i\Big\|=
\Big(\sum_{j=1}^{n_N}\Big(\sum_{i=1}^N\frac{|a_i|^p}{n_N^{p/q}}\Big)^{q/p}\Big)^{1/q}
=\Big(\frac{1}{n_N}\sum_{j=1}^{n_N}\Big(\sum_{i=1}^N|a_i|^p\Big)^{q/p}\Big)^{1/q}
=\Big(\sum_{i=1}^N|a_i|^p\Big)^{1/p}.$$

We  let  $Y=\text{\rm span}(y_i)$ and
define projections $T_X:\ell_q(\ell_p^N)\rightarrow X$ and
$T_Y:\ell_q(\ell_p^N)\rightarrow Y$ by
\begin{align*}
T_X\Big(\sum a_{(i,j)}e_{(i,j)}\Big)&=\sum_{i=1}^N\sum_{j=1}^{n_N/n_i}\sum_{s=1}^{n_i}\Big(\frac{1}{n_i}\sum_{k=1}^{n_i}a_{(i,k+(j-1)n_i)}\Big)e_{(i,s+(j-1)n_i)}\\
&=\sum_{i=1}^N \sum_{j=1}^{n_N/n_i} \Big(\frac1{n_i^{(q-1)/q}}  \sum_{k=1}^{n_i} a_{(i,k+(j-1)n_i)}\Big) x_{(i,j)}
\text{ and }\\
T_Y\Big(\sum a_{(i,j)}e_{(i,j)}\Big)&=\sum_{i=1}^N\sum_{s=1}^{n_N}\Big(\frac{1}{n_N}\sum_{k=1}^{n_N}a_{(i,k)}\Big)e_{(i,s)}
= \sum_{i=1}^N \Big(\frac1{n_N^{(q-1)/q}} \sum_{k=1}^{n_N}a_{(i,k)}\Big) y_i  .
\end{align*}
It is simple to check that $T_X$ and $T_Y$ are projections of
$\ell_q(\ell_p^N)$ onto $X$ and $Y$ respectively.  As $Y$ is a
subspace of $X$, we have that $T_Y$ restricted to $X$ is a
projection of $X$ onto $Y$.  Thus we just need to prove that there
exists a uniform constant $K$ such that $\|T_X\|,\|T_Y\|\leq K$.  We
note that $T_X=T_Y$ if, considering a more general case,
$n_i=n_N$ for all $1\leq i\leq N$.  Thus
proving there exists a constant $K$ independent of $(n_i)_{i=1}^N$
and $N$ such that $\|T_X\|\leq K$ will prove the inequality for
$\|T_Y\|$ as well.  We first consider the case that $p=q$.  In this
case, the basis $(e_{(i,j)})$ for the space $\ell_q(\ell_q^N)$ is
symmetric. The operator $T_X$ is then an averaging operator on a
space with a symmetric basis, and hence has norm
one.

Now if $q<p<\infty$, then $\ell_q(\ell_p^N)$ is an interpolation
space for the spaces $\ell_q(\ell_q^N)$ and $\ell_q(\ell_\infty^N)$.
Then by the vector valued Riesz-Thorin interpolation theorem \cite{BP} (see also \cite{C}) , if $\theta=\frac{q}{p}$ then
$$\|T_X\|_{\ell_q(\ell_p^N)}\leq
\|T_X\|^\theta_{\ell_q(\ell_q^N)}\|T_X\|^{1-\theta}_{\ell_q(\ell_\infty^N)}\leq\|T_X\|_{\ell_q(\ell_\infty^N)}.
$$
Thus if we prove that there exists a uniform constant $K$ such that
$\|T_X\|_{\ell_q(\ell_\infty^N)}\leq K$ then the result will follow
as well for all $1<q<p<\infty$. On the other hand, if $1\leq
p<q<\infty$ then $1<q'<p'\leq\infty$, with  $q'=q/(q-1)$ and
$p'=p/(p-1)$.  It is simple to check that our operator
$T_X:\ell_q(\ell_p^N)\rightarrow\ell_q(\ell_p^N)$ has adjoint
$T^*_X=T_X:\ell_{q'}(\ell_{p'}^N)\rightarrow\ell_{q'}(\ell_{p'}^N)$.
We thus have that
$\|T_X\|_{\ell_q(\ell_p^N)}=\|T^*_X\|_{\ell_{q'}(\ell_{p'}^N)}\leq
K$.

All that remains is to prove is that there exists a uniform constant
$K$ such that $\|T_X\|_{\ell_q(\ell_\infty^N)}\leq K$. This constant
$K$ will come from a  {\em discretization } of the classical
Hardy-Littlewood maximal operator \cite{HL}, which is defined as
\begin{equation}
n_u(g)(x)=\sup_{y<x<z}\frac{1}{z-y}\int_y^z|g(t)|dt\quad\textrm{ for }x\in\R\textrm{ and }g\in L^1_{loc}(\R).
\end{equation}

It is known that the operator $n_u$ is of strong type $(q,q)$ for
$1<q\leq\infty$ and for a proof of this see \cite[Theorem 8.9.1 and Corollary 8.9.1]{G}.  In other
words, there exists a constant $1\leq K<\infty$ such that
$$
\Big(\int|n_u(g)(x)|^q dx \Big)^{1/q}\leq K \Big(\int |g(x)|^q dx\Big)^{1/q} \text{ for all $g\in
L_q(\R)$.}$$
By applying this to step functions whose discontinuities
are contained in $\N$, we get the following inequality for
$\ell_q$.
\begin{equation}\label{E:2}
\Big(\sum_{j\in\N}\Big(\sup_{m\leq j\leq n}\frac{1}{n-m+1}\sum_{k=m}^n|a_k|\Big)^q\Big)^{1/q}\leq K
\Big(\sum_{j\in\N}|a_j|^q\Big)^{1/q}\quad\textrm{for all }(a_j)\in\ell_q.
\end{equation}

We  now  prove that
$\|T_X\|_{\ell_q(\ell_\infty^N)}\leq K$.  As $\ell_q(\ell_\infty^N)$ is
reflexive, the operator $T_X$ attains it's norm at an extreme point
of the ball $B_{\ell_q(\ell_\infty^N)}$.  The set of extreme points
of $B_{\ell_q(\ell_\infty^N)}$ is given by
$$Ext(B_{\ell_q(\ell_\infty^N)})=\Big\{\sum_{j=1}^\infty
a_j\sum_{i=1}^N\vp_{(i,j)}e_{(i,j)} : \vp_{(i,j)}=\pm1,(a_j)\in
S_{\ell_q}\Big\}.$$ Thus there exists constants $\vp_{(i,j)}=\pm1$ and
a sequence $(a_j)\in S_{\ell_q}$ such that
\begin{align*}
\|T_X\|_{\ell_q(\ell_\infty^N)}&=\Big\|T_X\Big(\sum_{j=1}^\infty a_j\sum_{i=1}^N\vp_{(i,j)}e_{(i,j)}\Big)\Big\|\\
&=\Big\|\sum_{i=1}^N\sum_{j=1}^{n_N/n_i}\sum_{s=1}^{n_i}\Big(\frac{1}{n_i}\sum_{k=1}^{n_i}\vp_{(i,k+(j-1)n_i)}a_{k+(j-1)n_i}\Big)e_{(i,s+(j-1)n_i)}\Big\|\\
&\leq\Big\|
\sum_{i=1}^N\sum_{j=1}^{n_N/n_i}\sum_{s=1}^{n_i}\Big(\frac{1}{n_i}\sum_{k=1}^{n_i}|a_{k+(j-1)n_i}|\Big)e_{(i,s+(j-1)n_i)}\Big\|\\
&\leq\Big\|
\sum_{i=1}^N\sum_{j=1}^{n_N/n_i}\sum_{s=1}^{n_i}\Big(\sup_{m\leq
1+(j-1)n_i\leq n}\frac{1}{n+1-m}\sum_{k=m}^{n}|a_{k}|\Big)e_{(i,s+(j-1)n_i)}\Big\|\\
&\le\Big\|\sum_{i=1}^N\sum_{j=1}^{n_N}\Big(\sup_{m\leq j\leq
n}\frac{1}{n+1-m}\sum_{k=m}^{n}|a_{k}|\Big)e_{(i,j)}\Big\|\\
&=\Big(\sum_{j=1}^{n_N}\Big(\sup_{m\leq j\leq n}\frac{1}{n+1-m}\sum_{k=m}^{n}|a_{k}|\Big)^q\Big)^{1/q}\\
&\leq K \Big(\sum_{j=1}^{n_N}|a_{k}|^q\Big)^{1/q}\quad\textrm{ by }(\ref{E:2}).
\end{align*}
\end{proof}

\begin{thm}\label{T:3}
The Banach space $(\oplus_{n=1}^\infty \ell_p^n)_{\ell_1}$ does not
have a greedy basis whenever $1\!<\!p\kleq\infty$.
\end{thm}

We recall that Bourgain, Casazza, Lindenstrauss, and Tzafriri
proved that the spaces $(\oplus_{n=1}^\infty \ell^n_2)_{\ell_1}$
and $(\oplus_{n=1}^\infty \ell^n_\infty)_{\ell_1}$ each have
unconditional bases which are unique up to permutation
\cite{BCLT}. In particular, these spaces cannot have a greedy
basis, as they each have an unconditional basis which is not
greedy. Thus we need only to  prove Theorem \ref{T:3} for the case
$1<p<\infty$. This is important for us as $\ell_p$ has non-trivial
type and cotype when $1<p<\infty$. We rely on the following
proposition, which was used in \cite{BCLT} to prove, among other
uniqueness results, that $(\oplus_{n=1}^\infty \ell^n_2)_{\ell_1}$
has a unique unconditional basis up to permutation.

\begin{propB}\label{P:2} {\rm  \cite[Proposition 2.1]{BCLT}  }
Let $V$ be a Banach lattice of type p and cotype q, for some $1\leq
p\leq q <\infty$ and let $C_c,C_u\geq1$ be constants.  There exists
a uniform constant $K\geq 1$ which satisfies the following
statement. If $Z$ is a $C_c$-complemented subspace of the direct sum
$X=(\oplus_{n=1}^\infty V)_{c_0}$ and $Z$ has a normalized basis
$(z_n)_{n=1}^k$ with unconditional constant $C_u$, then there exists
a partition of the integers $\{1,2,\ldots,k\}$ into mutually
disjoint subsets $\{\tau_s\}_{s=1}^r$ so that, for any choice of
scalars $\{\alpha_n\}_{n=1}^k$, we have
$$K^{-1}\max_{1\leq s\leq r}(\sum_{n\in\tau_s}|\alpha_n|^q)^{1/q}\leq\Big\|\sum_{n=1}^k \alpha_n
z_n\Big\|\leq K\max_{1\leq s\leq r}(\sum_{n\in\tau_s}|\alpha_n|^p)^{1/p}.
$$
\end{propB}

Using Proposition~B, we are now prepared to give a proof of
Theorem \ref{T:3}.
\begin{proof}[Proof of Theorem \ref{T:3}]

Let $1<p<\infty$.  To reach a contradiction, we assume that
$X=(\oplus_{n=1}^\infty \ell_p^n)_{\ell_1}$ has a normalized basis
$(x_i)_{i=1}^\infty$ which is $C_d$-democratic and $C_u$-unconditional for
some constants $C_d,C_u\geq1$. Let $(x^*_i)\subset \big(\oplus_{n}\ell_p^n\big)_{\ell_1}^*=
\big(\oplus_{n}\ell^n_{q}\big)_{\ell_\infty}$, with $\frac1p+\frac1{q}=1$, be the  biorthogonal functionals.
 We let $(e_{(i,n)})_{1\leq i\leq n<\infty}$ to be the unit vector basis for $X$ with biorthogonal
functionals $(e^*_{(i,n)})_{1\leq i\leq n<\infty}$. By a standard
perturbation argument we may assume that
$$\supp(x_j)=\{(i,n):n\in\N, i\le n, e^*_{(i,n)}(x_j)\not=0\}$$ is finite for each $j\in\N$.

We now fix $N\in\N$.  There exists $M_N\in\N$ such that
$(x_j)_{j=1}^{N}\subset (\oplus_{n=1}^{M_N} \ell_p^n)_{\ell_1}$, i.e.
$e^*_{(i,n)}(x_j)=0$ for all $n>M_n$ and  $i\leq n$.  Define
$(z^*_{(N,j)})_{j=1}^N\subset (\oplus_{n=1}^{M_N} \ell_{p}^n)_{\ell_1}^*$ by
$z^*_{(N,j)}=x^*_j|_{(\oplus_{n=1}^{M_N} \ell_{p}^n)_{\ell_1}}$ for all
$1\leq j\leq N$. As $(x_j)_{j=1}^\infty$ has basis constant at most
$C_u$, it is easy to show both that $(x^*_j)_{j=1}^{N}$ is
$C_u$-equivalent to $(z^*_{(N,j)})_{j=1}^{N}$, and that the span of
$(z^*_{(N,j)})_{j=1}^N$ is $C_u$-complemented in $(\oplus_{n=1}^{M_N}
\ell_p^n)_{\ell_1}^*$.
Indeed, let $J:\text{\rm span}(x_i)_{i=1}^N\rightarrow X$ be the inclusion map.  Then
$J^*:X^*\rightarrow \text{\rm span}(x^*_i)_{i=1}^N$ is a quotient map of norm 1.  Let
$H:\text{ \rm span}(x^*_i)_{i=1}^N\rightarrow \text{\rm span}(z^*_{(N,i)})_{i=1}^N$ be the isomorphism defined
by $H(x_i^*)=z^*_{(N,i)}$, which has norm at most $C_u$.  Then $H\circ J^*:X^*\rightarrow
\text{\rm span}(z^*_{(N,i)})_{i=1}^N$ is a quotient map of norm at most $C_u$.  We just need to check
that $H\circ J^*(z^*_{(N,i)})=z^*_{(N,i)}$.  Indeed,  for $x\in \text{\rm span}(x_j:j\le N)$,
$J^*(z^*_{(N,i)})(x)=z^*_{(N,i)}(x)=x_i^*(x)$.  Hence, $J^*(z^*_{(N,i)})=x^*_i$ and
$H\circ J^*(z^*_{(N,i)})=z^*_{(N,i)}$.  Thus $H\circ J^*$ is a projection of norm at most $C_u$.

We will be applying Proposition~B for the space
$V=\ell_q$ and consider the spaces
$\text{\rm span} (z^*_{(N,i)}:i\le N)$, $N\in\N$,  (in the natural way) to be $C_u$-complemented  subspaces of $c_0(\ell_q)$.
For the sake of
convenience, we denote $\underline{q}=\min(q,2)$ and $\bar{q}=\max(q,2)$.  We have thus defined $\underline{q}$ and $\bar{q}$ exactly so
that $\ell_q$ has type $\underline{q}$ and cotype $\bar{q}$.  By
Proposition~B, there exists a constant $K$ independent of
$N\in\N$, and there exists for all $N\in\N$ a partition of the
integers $\{1,2,\ldots,N\}$ into mutually disjoint subsets
$\{\tau^N_s\}_{s=1}^{r_N}$ so that, for any choice of scalars
$\{\alpha_n\}_{n=1}^N$, we have
\begin{equation}\label{E:1}
K^{-1}\max_{1\leq s\leq
{r_N}}(\sum_{n\in\tau^N_s}|\alpha_n|^{\bar{q}})^{1/{\bar{q}}}\leq\Big\|\sum_{n=1}^N
\alpha_n z^*_{(N,n)}\Big\|\leq K\max_{1\leq s\leq
{r_N}}(\sum_{n\in\tau^N_s}|\alpha_n|^{\underline{q}})^{1/{\underline{q}}}.
\end{equation}
We first consider the case that $\sup_{N\in\N}\max_{1\leq s\leq
r_N}|\tau^N_s|<\infty$.  In this case we have that if $N\in\N$ and
$(\alpha_i)_{i=1}^N\subset\R$ then
$$
\Big\|\sum_{i=1}^N \alpha_i x^*_i\Big\|\leq C_u \Big\|\sum_{i=1}^N \alpha_i
z^*_{(N,i)}\Big\|\leq C_u K \big( \sup_{M\in\N}\max_{1\leq s\leq r_M}|\tau^M_s|\big)
\max_{1\leq i\leq N}|\alpha_i|.
$$
Hence $(x^*_i)$ is equivalent to the unit vector basis  of $c_0$, which
implies that $(x_i)$ is equivalent to the unit vector basis of $\ell_1$.
This is a contradiction as $(x_i)$ is a basis for
$X=(\oplus_{n=1}^\infty \ell_p^n)_{\ell_1}$, and $X$ is not isomorphic
to $\ell_1$ as $1<p<\infty$.  We now consider the remaining case
that $\sup_{N\in\N}\max_{1\leq s\leq r_N}|\tau^N_s|=\infty$.

First
note that there exists a subsequence of $(x_i)_{i=1}^\infty$ which
is equivalent to the unit vector basis basis of $\ell_1$. Indeed,  there is a subsequence
$(x_j')$ which  converges in the $w^*$-topology (considering $X$ as the dual of $(\oplus_{n=1}^\infty \ell_q^n)_{c_0}$)
to some $x$. If $x\not= 0$, then it follows from the assumed unconditionality that $(x_j')$   is equivalent to the unit vector basis of $\ell_1$,
and if $x=0$, then  we can find a perturbation $(z'_j)$ of  a further subsequence, so that  for each $n\in\N$, there is at most one $j$ so that the
 $\ell_q^n$ component of $z'_j$  is not $0$. But this also implies that there is a subsequence equivalent to the unit vector basis of $\ell_1$.

 Thus as $(x_i)$ is
democratic,  there is a $C\ge 1$ so that $C\Big\|\sum_{i\in A}x_i\Big\|>|A|$ for all finite sets $A\subset\N$.  We choose
$N\in\N$ and $1\leq s\leq r_N$ such that
$|\tau^N_s|^{1/\bar{q}}>2K C C_u^2$. We have the following estimates.
\begin{align*}
\Big\|\sum_{i\in\tau^N_s} x_i\Big\|&\leq2C_u \Big(\sum_{i\in\tau^N_s} b_i x_i^*\Big)\Big(\sum_{i\in\tau^N_s} x_i\Big)
   \quad\textrm{ for some $(b_i)\in\coo$, with} \Big\|\sum_{i\in\tau^N_s} b_i x^*_i  \Big\|=1\\
&\leq 2C_u^2 \frac{\sum_{i\in\tau^N_s}|b_i|}{\Big\|\sum_{i\in\tau^N_s} b_i
z_{(N,i)}^*\Big\|}\\
&\leq 2KC_u^2\frac{\sum_{i\in\tau^N_s}|b_i|}{\sum_{i\in\tau^N_s} |b_i|^{\bar{q}})^{1/{\bar{q}}}}\\
&\leq2KC_u^2 |\tau^N_s|^{1-1/\bar{q}}\quad\textrm{ by
H\"{o}lder's inequality.}\\
\end{align*}
Combining this result with the inequality $|\tau^N_s|^{1/\bar{q}}>2K
C C_u^2$ gives the following contradiction. $$2K C C_u^2
|\tau^N_s|^{1-1/\bar{q}} <|\tau^N_s|<C \Big\|\sum_{i\in\tau^N_s}
x_i\Big\|\leq 2KC C_u^2 |\tau^N_s|^{1-1/\bar{q}}$$

As both possible cases result in a contradiction, we see that
$(\oplus_{n=1}^\infty \ell_p^n)_{\ell_1}$ cannot have a greedy basis
when $1<p<\infty$.
\end{proof}

Finally the case $q=\infty$ and $1\le p<\infty$ in Theorem \ref{T:2} is easy to handle.
\begin{prop}\label{P:2.3}
For $1\le p<\infty$ the space $(\oplus \ell_p^n)_{c_0}$ does not have a greedy basis.
\end{prop}
\begin{proof} Assume that $(x_n)$ is a greedy basis of $(\oplus \ell_p^n)_{c_0}$. Since $(x_n)$  is democratic and must contain a subsequence which is
equivalent to the unit vector basis of $c_0$ it follows that for some constant $C\ge 1$ that,
$$ \Big\| \sum_{n\in A} x_n\Big\|\le  C ,\text{ for all finite  $A\subset\N$.}$$
This together with the unconditionality of $(x_n)$ implies that $(x_n)$ is equivalent to the unit basis of $c_0$, which is a contradiction
since $(\oplus \ell_p^n)_{c_0}$ is not isomorphic to $c_0$.
\end{proof}

We have now finished the proof of Theorem \ref{T:2} and  determined which spaces of the form $(\oplus_{N=1}^\infty\ell_p^N)_{\ell_q}$
have a greedy basis.  We now turn to the proof of Theorem \ref{T:2a}.

We  rely on the
concept of greedy permutations developed by Albiac and Wojtaszczyk
\cite{AW}, which we recall here. Let $M(x)$ denote the subset of the
support of $x$ consisting of the largest coordinates of $x$ in
absolute value. We will say that a one to one map
$\pi:\text{\rm supp}(x)\rightarrow\N$ is a greedy permutation of $x$ if
$\pi(j)=j$ for all $j\in \text{\rm supp} (x)\setminus M(x)$ and if $j\in M(x)$
then, either $\pi(j) = j$ or $\pi(j)\not\in \text{\rm supp}(x)$.

\begin{defn}A basic sequence $(e_n)$ is defined to have property $(A)$ if for any
$x\in\text{\rm span}(e_i)$ we have
$$\|\sum_{n\in \text{\rm supp}(x)}e^*_n(x)e_n\|=\|\sum_{n\in \text{\rm supp}(x)}\theta_{\pi(n)}e^*_n(x)e_{\pi(n)}\|
$$
for all greedy permutations $\pi$ of $x$ and all sequences of signs
$(\theta_k)$ with $\theta_{\pi(n)}=1$ if $\pi(n)=n$.
\end{defn}
We recall
that $(e_n)$ is called {\em $C$-suppression unconditional}, for some $C\ge 1$, if  for
 any  $(a_i)\subset c_{00}$ and any $A\subset \N$,
  $$\Big\|\sum_{i\in A} a_i e_i\Big\| \le C\Big\|\sum_{i\in\N} a_i e_i\Big\|.$$
\begin{thmC}\label{T:10}{\rm  \cite[Theorem 3.4]{AW}  }
A basic sequence $(e_n)$ is 1-greedy if and only if $(e_n)$ is
1-suppression unconditional and satisfies property $(A)$.
\end{thmC}

%%%%%%%
\begin{proof}[Proof of Theorem \ref{T:2a}]
We first consider the case that $1<p<\infty$. We show that if
$A\subset\N$ is any finite set, then $\|\sum_{i\in A} a_i
x_i\|=(\sum_{i\in A} |a_i|^p)^{1/p}$ for all $(a_i)\in\coo$. This is
trivial if $|A|=1$ as $(x_i)$ is normalized. We now assume that the
equality holds for $|A|\leq k$ for some $k\geq1$.  Let
$(a_i)\in\coo$ and $A\subset\N$ such that $|A|=k+1$. Choose $N\in A$
such that $|a_N|=\max_{i\in A}|a_i|$.  We define
$\pi_j:A\rightarrow\N$ by $\pi_j(N)=j$ and $\pi_j(n)=n$ for all
$n\neq N$.  The map $\pi_j$ is a greedy permutation whenever
$j\not\in A$, and hence by Theorem~C we have the following
equalities.
\begin{align*}
\Big\|\sum_{i\in A}a_i x_i\Big\|&=\Big\|\sum_{i\in A,i\neq N}a_i x_i +a_N
x_j\Big\|\quad\textrm{for all }j\not\in A\\
&=\lim_{j\rightarrow\infty}\Big\|\sum_{i\in A,i\neq N}a_i x_i +a_N
x_j\Big\|\\
&=\Big(\Big\|\sum_{i\in A,i\neq N}a_i x_i\Big\|^p +|a_N|^p\Big)^{1/p},\quad\textrm{as }(x_j)\textrm{ is normalized and weakly null,}\\
&=\Big(\sum_{i\in A}|a_i|^p\Big)^{1/p},\quad\textrm{by the induction hypothesis}.\\
\end{align*}
This finishes the proof of the induction step, and, thus, the proof of our claim.

The case $p=\infty$, in which case we consider the $c_0$-sum of the $E_n$, works similarly,
as every normalized unconditional sequence in
$(\sum E_n)_{c_0}$ must be weakly null.

We now consider the $\ell_1$
case. Let $(x_i)$ be a 1-greedy basis for $(\sum E_n)_{\ell_1}$.  If
$(x_i)$ is $w^*$-null with respect to the $w^*$ topology given by
$(\sum E^*_n)_{\ell_\infty}$, then the proof that $(x_i)$ is 1-equivalent to
the unit vector basis for $\ell_1$ is the same as the previous case
$1<p<\infty$. If $(x_i)$ is not $w^*$-null, then $(x_i)$ has a
subsequence $(x_{k_i})$ which converges $w^*$ to some non-zero
$x\in(\sum E_n)_{\ell_1}$. Hence $(x_{k_i}-x)$ is $w^*$-null.  This
implies that
$\lim_{i\rightarrow\infty}\|y+x_{k_i}-x\|=\|y\|+\lim_{i\rightarrow\infty}\|x_{k_i}-x\|$
for all $y\in(\sum E_n)_{\ell_1}$. We use this to achieve the
following equalities.
\begin{align*}
\lim_{n\rightarrow\infty}\lim_{i\rightarrow\infty}\|x_{k_n}-x_{k_i}\|&=\lim_{n\rightarrow\infty}\lim_{i\rightarrow\infty}\|(x_{k_n}-x)-(x_{k_i}-x)\|\\
&=\lim_{n\rightarrow\infty}\|x_{k_n}-x\|+\lim_{i\rightarrow\infty}\|x_{k_i}-x\|=2\lim_{i\rightarrow\infty}\|x_{k_i}-x\|.
\end{align*}
Furthermore,
\begin{align*}
\lim_{n\rightarrow\infty}\lim_{i\rightarrow\infty}\|x_{k_n}+x_{k_i}\|&=\lim_{n\rightarrow\infty}\lim_{i\rightarrow\infty}\|(x_{k_n}-x)+(x_{k_i}-x)+2x\|\\
&=\lim_{n\rightarrow\infty}\|x_{k_n}-x+2x\|+\lim_{i\rightarrow\infty}\|x_{k_i}-x\|\\
&=2\|x\|+\lim_{n\rightarrow\infty}\|x_{k_n}-x\|+\lim_{i\rightarrow\infty}\|x_{k_i}-x\|=2\|x\|+2\lim_{i\rightarrow\infty}\|x_{k_i}-x\|.
\end{align*}
As $(x_i)$ is 1-greedy, we must have, by Theorem~C, that
$\lim_{n\rightarrow\infty}\lim_{i\rightarrow\infty}\|x_{k_n}-x_{k_i}\|=\lim_{n\rightarrow\infty}\lim_{i\rightarrow\infty}\|x_{k_n}+x_{k_i}\|$.
This however implies that $\|x\|=0$, which is a contradiction with
our assumption that $x$ is non-zero.
\end{proof}

\end{document}